\documentclass[11pt]{amsart}
\usepackage{amssymb,amsmath,amsfonts,amscd,euscript}

\newcommand{\nc}{\newcommand}

\numberwithin{equation}{section}
\newtheorem{thm}{Theorem}[section]
\newtheorem{prop}[thm]{Proposition}
\newtheorem{lem}[thm]{Lemma}
\newtheorem{cor}[thm]{Corollary}
\theoremstyle{remark}
\newtheorem{rem}[thm]{Remark}

\newtheorem{example}[thm]{Example}
\newtheorem{dfn}[thm]{Definition}

\newtheorem{conj}[thm]{Conjecture}

\nc{\gl}{\mathfrak{gl}}
\nc{\GL}{\mathfrak{GL}}
\nc{\g}{\mathfrak{g}}
\nc{\gh}{\widehat\g}
\nc{\h}{\mathfrak{h}}
\nc{\la}{\lambda}
\nc{\al}{\alpha }
\nc{\be}{\beta }
\nc{\ve}{\varepsilon }
\nc{\om}{\omega }

\nc{\ta}{\theta}
\nc{\ch}{{\mathop {\rm ch}}}
\nc{\Tr}{{\mathop {\rm Tr}\,}}
\nc{\Id}{{\mathop {\rm Id}}}
\nc{\ad}{{\mathop {\rm ad}}}
\nc{\bra}{\langle}
\nc{\ket}{\rangle}
\nc{\x}{{\bf x}}
\nc{\bm}{{\bf m}}
\nc{\bs}{{\bf s}}
\nc{\bp}{{\bf p}}
\nc{\bc}{{\bf c}}
\nc{\pa}{\partial}
\nc{\ld}{\ldots}
\nc{\cd}{\cdots}
\nc{\hk}{\hookrightarrow}
\nc{\T}{\otimes}
\nc{\gr}{\mathrm{gr}}
\nc{\ov}{\overline}

\nc{\cO}{\mathcal O}
\nc{\msl}{\mathfrak{sl}}
\nc{\mgl}{\mathfrak{gl}}
\nc{\U}{\mathrm U}
\nc{\V}{\EuScript V}
\nc{\cL}{\mathcal{L}}
\nc{\Res}{\mathrm{Res\ }}

\newcommand{\bC}{{\mathbb C}}
\newcommand{\bZ}{{\mathbb Z}}

\newcommand{\bP}{{\mathbb P}}
\newcommand{\bG}{{\mathbb G}}

\newcommand{\fg}{{\mathfrak g}}
\newcommand{\fb}{{\mathfrak b}}

\newcommand{\fn}{{\mathfrak n}}

\newcommand{\Fl}{\EuScript{F}}

\begin{document}

\title[Degenerate $SL_n$: representations and flag varieties]
{Degenerate $SL_n$: representations and flag varieties}

\author{Evgeny Feigin}
\address{Evgeny Feigin:\newline
Department of Mathematics,\newline
National Research University Higher School of Economics,\newline
Vavilova str. 7, 117312, Moscow, Russia,\newline
{\it and }\newline
Tamm Theory Division, Lebedev Physics Institute
}
\email{evgfeig@gmail.com}

\begin{abstract}
The degenerate Lie group is a semidirect product of the Borel subgroup with the normal 
abelian unipotent subgroup.
We introduce a class of the highest weight representations of the degenerate group of type A, generalizing
the PBW-graded representations of the classical group. Following the classical construction
of the flag varieties, we consider the closures of the orbits of the abelian
unipotent subgroup in the projectivizations of the representations. We show that the degenerate flag
varieties $\Fl^a_n$ and their desingularizations $R_n$ can be obtained via this construction. 
We prove that the coordinate ring of $R_n$ is isomorphic to the direct sum of duals of the highest 
weight representations of the degenerate group. In the end, we state several conjectures on the structure
of the highest weight representations.   
\end{abstract}

\maketitle

\section*{Introduction}
Let $\g$ be a simple Lie algebra with the Borel subalgebra $\fb$ and let $G,B$ be the corresponding groups. 
The irreducible highest weight representation of $\g$ play fundamental role
in the algebraic and geometric Lie theory. In particular, the generalized flag varieties for $G$
can be realized as $G$-orbits inside the projectivizations of these modules. 
Let $\g^a$ and $G^a$ be the degenerate Lie algebra and the degenerate Lie group 
(see \cite{Fe1},\cite{Fe2}, \cite{FF}, \cite{FFiL}, \cite{PY}, \cite{Y}). 
The Lie algebra
is a sum of the subalgebra $\fb$ and of the abelian ideal $\g/\fb$ with the adjoint action of $\fb$ on
$\g/\fb$. The Lie group is the semidirect
product of the Borel subgroup with the normal abelian unipotent subgroup $\exp(\g/\fb)$.
In this paper we are concerned with the following question: what are the analogues of the finite-dimensional
representations of $\g$ and of the flag varieties in the degenerate situation? 

In this paper we only study the type $A$ case, so from now on $\g=\msl_n$ and $G=SL_n$. Recall that in this
case the fundamental representations $V_{\omega_k}$ are labeled by a number $k=1,\dots,n-1$ and one has
$V_{\omega_k}\simeq \Lambda^k(\bC^n)$. Now let $\la=\sum_{i=1}^{n-1} m_i\om_i$ be a dominant weight, $m_i$ are
non-negative integers. Then the corresponding highest weight representation $V_\la$ sits inside
the tensor product of $V_{\om_i}$'s, where each factor appears exactly $m_i$ times. The image of the
embedding is nothing but the $\g$-span of the tensor product of highest weight vectors. It has been shown
in \cite{FFoL1},\cite{FFoL2}, \cite{Fe1} that the degenerate Lie algebra $\g^a$ naturally acts on the
PBW-graded representations $V^a_\la$ (the associated graded to $V_\la$ with respect to the PBW filtration). 
In addition, $V^a_\la$ still sits inside the tensor product of
$V^a_{\omega_i}$'s in the same way as $V_\la$ in the tensor product of $V_{\om_i}$'s. 
The natural question is: are there natural representations of $\g^a$ 
different from $V^a_\la$? It turns out that this question is important for the study of the PBW-degeneration
of flag varieties, see \cite{Fe1}, \cite{Fe2}, \cite{FF}. Let us recall basic steps here.

Let $\fn^-$ be the nilpotent subalgebra such that $\g=\fb\oplus\fn^-$. We denote by $N^-$ the corresponding
unipotent subgroup of $G$. 
Let $(N^-)^a\simeq \exp(\g/\fb)$ be the abelian unipotent subgroup of $G^a$. This group is isomorphic
to the product of $\dim\fn^-$ copies of the group $\bG_a$ -- the additive group of the field.
Let  $v_\la\in V_\la$ be a highest weight vector.  
Recall that the generalized flag variety $\Fl_\la$ is defined as the $G$-orbit of the highest weight
line $\bC v_\la$ in $\bP(V_\la)$ (see \cite{FH}, \cite{Fu}, \cite{K}). 
The corresponding degenerate flag variety is defined as the closure of
the  $(N^-)^a$-orbit of $\bC v_\la$ in $\bP(V^a_\la)$ (see \cite{Fe1}, \cite{Fe2}). 
For general $\la$, this is a normal singular projective
variety enjoying explicit description in terms of linear algebra. Let us restrict here to the case of 
regular dominant $\la$ (i.e. all $m_i$ are positive). Then all such flag varieties are isomorphic
and we denote the corresponding variety by $\Fl^a_n$. $\Fl^a_n$ can be explicitly realized as the variety of
collections $(V_i)_{i=1}^{n-1}$ of subspaces of an $n$-dimensional space
subject to certain conditions. We note that the entries $V_i$ correspond to fundamental weights $\omega_i$.

In \cite{FF} a desingularization $R_n$ for $\Fl^a_n$ was constructed in terms of linear algebra. These 
$R_n$ are smooth projective algebraic varieties, which are Bott towers, i.e. can be constructed as succesive
fibrations with fibers $\bP^1$. A point in $R_n$ is a collection $(V_{i,j})$ with $1\le i\le j\le n-1$. 
This suggests that the "fundamental" representations of the Lie algebra $\g^a$ are in one-to-one correspondence
with the set of positive (non-necessarily simple) roots. In fact, it turns out that to each positive root
$\al$ one can attach a representation $M_\al$. These representations are highest weight in a sense that 
$M_\al$ is generated from a highest weight vector $v_\al$ by the action of the symmetric algebra of $\fn^-$
(for example, for simple roots $\al_i$ one has $M_{\al_i}\simeq V^a_{\om_i}$). 
In addition, $R_n$ can be embedded into the product of $\bP(M_\al)$ (over all positive roots
$\al$) as the closure of the $(N^-)^a$-orbit through the product of the highest weight lines. Similar to
the classical situation, fundamental representations give us a way to construct a large class of $\g^a$-modules.
Namely, given a collection of non-negative integers $\bm=(m_\al)$, $\al$ -- positive root, we consider the
$\g^a$-module
\[
M_\bm=S^\bullet (\fn^-)\cdot v_\bm\subset \bigotimes M_{\al}^{\T m_\al},\ \ 
v_\bm=\bigotimes v_{\al}^{\T m_\al}.
\]  
In particular, if $m_\al=0$ for non-simple $\al$, then $M_\bm\simeq V^a_\la$ for 
$\la=\sum m_{\al_i} \om_i$ ($\al_i$ are simple roots). 
We give more examples of such modules and formulate several conjectures concerning the structure off 
$M_\bm$. 

By definition, for any two collections $\bm^1$ and $\bm^2$, we have a $\g^a$-equivariant embedding
$M_{\bm^1+\bm^2}\subset M_{\bm^1}\T M_{\bm^2}$, 
$v_{\bm^1+\bm^2}\mapsto v_{\bm^1}\T v_{\bm^2}$. Dualizing, we obtain an algebra 
$\bigoplus_\bm M_\bm^*$. Our main theorem is as follows:
\begin{thm}
The coordinate ring of $R_n$ is isomorphic to $\bigoplus_\bm M_\bm^*$.
\end{thm} 
Our main tool is the explicit form of the Pl\"ucker-type relations in the coordinate ring of $R_n$.

Finally we note that one can naturally attach to a module $M_\bm$ the corresponding "flag variety".
Namely, let $\Fl(M_\bm)\subset\bP(M_\bm)$ be the closure of the orbit $(N^-)^a \cdot \bC v_\bm$.These
are the so-called $\bG_a$-varieties (see \cite{A}, \cite{HT}). We note that if all $m_\al$ are positive, then
$\Fl(M_\bm)\simeq R_n$ and if all $m_\al$ but $m_{\al_i}$ vanish, then $\Fl(M_\bm)\simeq \Fl^a_n$.      

Our paper is organized as follows:\\
In Section 1 we settle notation and recall main definitions and constructions.\\  
In Section 2 we state our results and provide examples.\\
Section 3 is devoted to the proofs and in Section 4 several conjectures are stated. 

\section{Notation and main objects}
\subsection{Classical story.} Let $\g$ be a simple Lie algebra with the Cartan decomposition 
$\g=\fn\oplus\h\oplus\fn^-$. Let $\fb=\fn\oplus\h$ be the Borel subalgebra. We denote by
$\Phi^+$ the set of positive roots of $\g$ and by $\al_1,\dots,\al_l\in\Phi^+$ the set of simple
roots. We sometimes write $\al>0$ instead of $\al\in\Phi^+$. For $\al>0$ we denote by $f_\al\in\fn^-$ a 
weight $-\al$ element. Thus we have $\fn^-=\bigoplus_{\al>0} \bC f_\al$. Let $G$, $B$, $N$,
$N^-$ and $T$ be the Lie groups corresponding to the Lie algebras $\g$, $\fb$, $\fn$, $\fn^-$ and $\h$.

Let $\om_1,\dots,\om_l$ be the fundamental weights. The fundamental weights and simple roots are
orthogonal with respect to the Killing form $(\cdot,\cdot)$ on $\h^*$: $(\om_i,\al_j)=\delta_{i,j}$.
A dominant integral weight $\la$ is
given by $\sum_{i=1}^l m_i\om_i$, $m_i\in\bZ_{\ge 0}$. 
For a dominant integral $\la$ let
$V_\la$ be the finite-dimensional irreducible highest weight $\g$-module with highest weight $\la$ and 
a highest weight vector $v_\la$ such that $\fn v_\la=0$, $h v_\la=\la(h) v_\al$ ($h\in\h$) and
$V_\la=\U(\fn^-) v_\la$. 

The (generalized) flag varieties for $G$ are defined as quotient $G/P$ by the parabolic subgroups. These
varieties play crucial role in the geometric representation theory. An important feature of the
flag varieties is that they can be naturally embedded into the projectivization of the highest weight
modules. Namely, let $\la$ be a dominant weight such that the stabilizer of the line 
$[v_\la]$ in $G$ is equal to $P$. Here and below for a vector $v$ in a vector space $V$ we denote by 
$[v]\in\bP(V)$ the line spanned by $v$.
Then one gets the embedding $G/P\subset \bP(V_\la)$ as the $G$-orbit of
the highest weight line $[v_\la]$. For a dominant weight $\la$ we denote by $\Fl_\la\subset \bP(V_\la)$
the orbit $G[v_\la]$ of the highest weight line. These are smooth projective algebraic varieties.
 It is clear that $\Fl_\la\simeq \Fl_\mu$
if and only if for all $i$ $(\la,\om_i)=0$ is equivalent to $(\mu,\om_i)=0$.         

\subsection{Degenerate version.}
Let $\g^a$ be the degenerate Lie algebra defined as a direct sum $\fb\oplus \g/\fb$ of the Borel subalgebra
$\fb$ and abelian ideal $\g/\fb$ (see \cite{Fe1},\cite{Fe2}). The algebra $\fb$ acts on $\g/\fb$ via 
the adjoint action.
We denote the space $\g/\fb$ by $(\fn^-)^a$ ($a$ is for abelian). Let $(N^-)^a=\exp (\fn^-)^a$ be the abelian 
Lie group, which is nothing but the product of $\dim \fn^-$ copies of the group $\bG_a$ --  the
additive group of the field. Let $G^a$ be the semidirect product $B\ltimes (N^-)^a$ of the subgroup $B$
and of the normal abelian subgroup $(N^-)^a$ (the action of $B$ on $(N^-)^a$ is induced by the action of $B$ on
$\fn^-$ by conjugation).  

Similar to the classical situation, we say that  a $\g^a$-module $M$ is a highest weight module if
there exists $v\in M$ such that $\fn v=0$, the line $[v]$ is $\h$-stable and $M=S(\fn^-)v$, where $S(\fn^-)$ denotes 
the symmetric algebra of $\fn^-$, which is isomorphic to the polynomial ring $\bC[f_\al]_{\al>0}$. 
It is clear that the highest weight $\g^a$-modules are in one-to-one correspondence with the $\fb$-invariant
ideals in $\bC[f_\al]_{\al>0}$. Namely, $M$ defines the  annihilating ideal 
$\mathrm{Ann} M\subset \bC[f_\al]_{\al>0}$ and a
$\fb$-invariant ideal $I$ produces the $\g^a$-module $\bC[f_\al]_{\al>0}/I$. 
In \cite{FFoL1}, \cite{FFoL2} the modules $V_\la^a$ were studied for dominant integral $\la$. 
The module $V_\la^a$ is the associated graded module to $V_\la$ with respect to the PBW filtration.
We note that all $V_\la^a$ are highest weight modules. However, these are not all highest weight $\g^a$-modules. Below we study a wider class of representations for $\g=\msl_n$ (still not exhaustive). We note that all the 
highest weight $\g^a$-modules which show up so far, are graded compatibly with with the grading by the
total degree on $\bC[f_\al]_{\al>0}$. Summarizing, we put forward the following definition:
\begin{dfn}\label{hwm}
A $\g^a$-module $M=\bigoplus_{s\ge 0} M(s)$ is called a graded highest weight module 
if the following holds:
\begin{itemize}
\item $M(0)=\bC v, \fn v=0,\ \h v\subset M(0)$,
\item $M=\bC[f_\al]_{\al>0} v$,
\item $f_\al M(s)\subset M(s+1), \al>0$. 
\end{itemize} 
\end{dfn}  
We note that the definition implies that $\fb M(s)\subset M(s)$. In what follows we say that elements of
$M(s)$ have PBW-degree $s$.

Now we want to define degenerate version of the flag varieties. It is not reasonable to take the quotient
of $G^a$ now (for example, $G^a/B$ is simply an affine space). Instead we use the highest weight representations.
Let $M$ be a graded highest weight $\g^a$-module with a highest weight vector $v$. Then we put forward the following
definition:
\begin{dfn}
The variety $\Fl(M)$ is the closure of the $(N^-)^a$-orbit of the highest weight line:
\[
\Fl(M)=\ov{(N^-)^a\cdot [v]}\subset \bP(M).
\]
\end{dfn} 
The degenerate flag varieties $\Fl^a_\la$ are isomorphic to $\Fl(V_\la^a)$.
We note that the group $(N^-)^a$ is isomorphic to the product of $\dim\fn^-$ copies of the group
$\bG_a$ -- the additive group of the field. Therefore, all the varieties $\Fl(M)$ are the so-called
$\bG_a$-varieties, i.e. the compactifications of the abelian unipotent group (see \cite{A}, \cite{HT}). 

\subsection{The type $A$ case.}
From now on $\g=\msl_n$ and $G=SL_n$. In this case the set of positive roots $\Phi^+$ consists of the roots
\[
\al_{i,j}=\al_i + \al_{i+1} +\dots +\al_j,\ 1\le i\le j\le n-1.
\]
In what follows we use the shorthand notation $f_{i,j}$ for $f_{\al_{i,j}}$.
The following proposition is proved in \cite{FFoL1} ($\circ$ denotes the adjoint action of $\fb$ on 
$S(\fn^-)\simeq S(\fg/\fb)$).
\begin{prop}\label{ffl}
For a dominant weight $\la$ the module $V_\la^a$ is isomorphic to the quotient $S(\fn^-)/I_\la$, where
$I_\la$ is the ideal generated by the subspace
\[
\U(\fb)\circ \mathrm{span}(f_\beta^{(\beta,\la)+1}, \beta\in\Psi^+). 
\]
\end{prop}

It is proved in \cite{Fe2} that the degenerate flag varieties enjoy the following explicit realization.
For simplicity, we describe the case of the complete flags here. So let $\la$ be a regular dominant weight,
i.e. $(\la,\om_i)>0$ for all $i$. Fix an $n$-dimensional vector space $W$ and a basis $w_1,\dots,w_n$ of $W$.
The operators $f_{i,j}$ act on $W$ by the standard formula $f_{i,j}w_k=\delta_{i,k}w_{j+1}$.
Let $pr_k:W\to W$ be a projection along the $k$-th basis vector, i.e. $pr_k w_k=0$ and $pr_k w_j=w_j$ for
$j\ne k$. All the flag varieties $\Fl^a_\la$ are isomorphic ($\la$ is regular) and we denote this variety 
by $\Fl^a_n$. Then $\Fl^a_n$ consists of collections
$(V_1,\dots,V_{n-1})$ of subspaces of $W$ such that $\dim V_i=i$ and $pr_{i+1}V_i\subset V_{i+1}$.
The varieties $\Fl^a_n$ are singular (starting from $n=3$) projective algebraic varieties, naturally embedded
into the product of Grassmannians. A desingularization $R_n$ for $\Fl^a_n$ was constructed in \cite{FF}
(see also \cite{FFiL} for the symplectic version).
Namely, $R_n$ consists of collections $(V_{i,j})_{1\le i\le j\le n-1}$ of subspaces of $W$ subject to the following
conditions:
\begin{itemize}
\item $\dim V_{i,j}=i,\ V_{i,j}\subset \mathrm{span}(w_1,\dots,w_i,w_{j+1},\dots,w_n)$,
\item $V_{i,j}\subset V_{i+1,j},\ 1\le i<j\le n-1$,\
\item $pr_{j+1}V_{i,j}\subset V_{i,j+1},\ 1\le i\le j<n-1$.
\end{itemize}  
One can show that the variety $R_n$ is a Bott tower, i.e. there exists a tower 
$$R_n=R_n(0)\to R_n(1)\to R_n(2)\to\dots \to R_n(n(n-1)/2)=pt$$
such that each map $R_n(l)\to R_n(l+1)$ is a $\bP^1$-fibration. Hence, $R_n$ is smooth. The surjection
$R_n\to\Fl^a_n$ is given by $(V_{i,j})_{1\le i\le j\le n-1}\mapsto (V_{i,i})_{i=1}^n$.  
Let $W_{i,j}=\mathrm{span}(w_1,\dots,w_i,w_{j+1},\dots,w_n)$. Then from the definition of $R_n$ we obtain
\begin{equation}
R_n\subset \prod_{1\le i\le j\le n-1}\bP(\Lambda^iW_{i,j}).
\end{equation}

\section{Representations and  coordinate rings}
In this section we state our results and provide examples. The proofs are given in the next section.
Throughout the section $\g=\msl_n$.

\subsection{Representations.}
For a positive root $\al=\al_{i,j}$ let $M_\al$ be the $\g^a$-module defined as the quotient 
$\bC[f_\beta]_{\beta>0}/I_\al$, where $I_\al$ is the ideal generated by the subspace 
\[
U(\fb)\circ \langle f^2_\beta, \beta\ge \al; 
f_\beta, \beta\not\ge \al\rangle.
\]
(Here $\circ$ denotes the adjoint action).
We denote the highest weight vector (the image of $1$) in $M_\al$ by $v_\al$.

\begin{example}
Let $\al=\al_1+\dots +\al_{n-1}$ be the highest root. Then $M_\al$ is two-dimensional space with a basis 
$v_\al$, $f_\al v_\al$. All the elements $f_\beta$, $\beta\ne\al$ as well as $\fb\subset\g^a$, 
act trivially on $M_\al$.
\end{example}

\begin{lem}
If $\al$ is a simple root $\al_i$, then $M_{\al_i}\simeq V_{\omega_i}^a$ as $\g^a$-modules.
\end{lem}
\begin{proof}
Follows from Proposition \ref{ffl}.
\end{proof}

\begin{lem}
Let $\al=\al_{i,j}$. Then $M_\al\simeq \Lambda^i W_{i,j}$.
\end{lem}
\begin{proof}
Let us define the structure of $\g^a$-module on $\Lambda^i W_{i,j}$ in the following way.
Consider $V^a_{\om_i}$. Then as a vector space this module is isomorphic to $\Lambda^i W$.
It is easy to see that the subspace $U$ spanned by the vectors
\[
w_{l_1}\wedge\dots \wedge w_{l_i}, \exists r: i<l_r\le j
\]
is $\g^a$-invariant. Now clearly we can identify the quotient of $\Lambda^i W/U$ with
$\Lambda^i W_{i,j}$. In addition, this quotient is isomorphic to $M_{\al_{i,j}}$ accordiing to 
Proposition \ref{ffl}. 
\end{proof}

\begin{rem}
$1)$. We note that $M_\al=\bigoplus_{s\ge 0} M_\al(s)$ (see Definition \ref{hwm}) and $M_\al(s)$ is spanned
by the vectors
\[
w_{l_1}\wedge\dots\wedge w_{l_i},\ l_a\in\{1,\dots,i,j+1,\dots,n\},\ \#\{a:\ l_a>j\}=s.
\] 
$2)$. The operators $f_{a,b}$ act trivially on $M_{\al_{i,j}}$ unless $a\le i\le j\le b$.
\end{rem}

The modules $M_{\al_{i,j}}$ play a role of the fundamental representations for the Lie algebra $\g^a$.
We now define a larger class of representations "generated" by the fundamental ones. 
Let $\bm=(m_{i,j})_{i\le j}$ be a collection of non-negative integers. Here and in what follows we write
$i\le j$ (or $a\le b$) instead of $1\le i\le j\le n-1$ (or $1\le a\le b\le n-1$) if it does not lead to any
confusion. We put forward the following definition
\begin{dfn}
The $\g^a$-module $M_\bm$ is the top (Cartan) component in the tensor product of modules $M_{\al_{i,j}}$, i.e.:
\[
M_\bm=S(\fn^-)\cdot \bigotimes_{i\le j} v_{\al_{i,j}}^{\T m_{i,j}}\subset 
\bigotimes_{i\le j} M_{\al_{i,j}}^{\T m_{i,j}}.
\]  
\end{dfn}
We denote the highest weight vector of $M_\bm$, which is the tensor product of the vectors $v_{\al_{i,j}}$,
by $v_\bm$.

\begin{example}
Let $m_{i,j}=0$ unless $i=j$. Then according to \cite{FFoL1} we have $M_\bm\simeq V^a_\la$, where
$\la=\sum_{i=1}^{n-1} m_{i,i}\om_i$. 
\end{example}

\begin{example}
Let $m_{1,j}=1$ for $j=1,2,\dots,n-1$ and $m_{i,j}=0$ if $i>1$. Then $f_{i,j} v_\bm=0$ unless $i=1$ and
it is easy to see that $f_{1,1}^{k_1}\dots f_{1,n-1}^{k_{n-1}}v_\bm\ne 0$ if and only if 
\begin{equation}\label{Cat}
k_1\le 1,\ k_1+k_2\le 2, \dots, k_1 +\dots +k_{n-1}\le n-1.
\end{equation}
Moreover, the monomials satisfying conditions above are linearly independent. 
We note that the number of collections $(k_1,\dots,k_{n-1})$ subject to the restrictions \eqref{Cat}
is equal to the Catalan number $C_n$. In addition, the dimensions of the homogeneous components 
$\dim M_\bm(s)$ 
are given by the entries of the Catalan triangle (see e.g. http://mathworld.wolfram.com/CatalansTriangle.html).
\end{example}

The following Lemma explains the importance of the above defined modules. We say that a collection
$\bm$ is regular if $m_{i,j}>0$ for all $i\le j$.
\begin{prop}
For any regular $\bm$, the variety $R_n$ is embedded into the projective space $\bP(M_\bm)$ 
as the closure of the $SL_n^a$-orbit or $(N^-)^a$-orbit through the highest weight line:
\[
R_n=\overline{SL_n^a\cdot [v_\bm]}=\overline{(N^-)^a\cdot [v_\bm]}\subset \bP(M_\bm).
\]
\end{prop}
\begin{proof}
It suffices to show that our proposition holds if all $m_{i,j}=1$. We note that the orbits 
$SL_n^a\cdot [v_\bm]$ and $(N^-)^a\cdot [v_\bm]$ coincide and are embedded into the product of Grassmannians
$\prod_{i\le j} Gr_i(W_{i,j})$. Recall that we have fixed a basis $w_1,\dots,w_n$ in $W$. For a collection
$$L=(l_1,\dots,l_i)\subset \{1,\dots,i,j+1,\dots,n\}$$
 and an element $V_{i,j}\in Gr_i(W_{i,j})$, let 
$X_L(V_{i,j})$ be the Pl\"ucker coordinate of $V_{i,j}$. 
We claim that the orbit $(N^-)^a [v_\bm]$ consists exactly of the collections of subspaces
$(V_{i,j})_{i\le j}\in R_n$ such that $X_{1,\dots,i}(V_{i,j})\ne 0$ (see Lemma 2.11 and Lemma 2.12 of \cite{FF}).
This implies the proposition.
\end{proof}
In what follows we consider $R_n$ as being embedded into $\bP(M_\bm)$ with $m_{i,j}=1$ for all $i\le j$,
or, equivalently, being embedded into $\prod_{\al>0} \bP(M_\al)$. We denote such a collection by ${\bf 1}$.

\subsection{Coordinate ring.}
We denote  by $X^{(i,j)}_L=X^{(i,j)}_{l_1,\dots,l_i}$ the dual (Pl\"ucker) coordinates in $W_{i,j}$ with respect
to $w_{l_1}\wedge\dots\wedge w_{l_i}$, where 
$1\le l_1<\dots <l_i\le n$ and $L\subset \{1,\dots,i,j+1,\dots,n\}$. For a permutation 
$\sigma\in S_n$ we set 
\[
X^{(i,j)}_{l_{\sigma(1)},\dots,l_{\sigma(i)}}=(-1)^\sigma X^{(i,j)}_{l_1,\dots,l_i}.
\]
The PBW-degree of a variable $X^{(i,j)}_L$ is defined as the number of entries $l_a$ such that $l_a>j$.
\begin{rem}
The variables of the PBW-degree zero are of the form $X^{(i,j)}_{1,\dots,i}$. The variables of PBW-degree one
are of the form $X^{(i,j)}_{1,\dots,r-1,r+1,\dots,i,m}$ for some $r\le i\le j<m$.
\end{rem}

Let us give explicit description of the $SL_n^a$-orbit through the highest weight line in 
$M_{\bf 1}$.   
For a collection of complex numbers $(c_{a,b})_{a\le b}$ consider the element 
$$\exp(\sum_{a\le b} c_{a,b} f_{a,b})\in (N^-)^a.$$ 
Let $L=(1\le l_1<\dots <l_i\le n)$ and let 
$d$ be a number such that  $l_d\le i < l_{d+1}$. We consider the numbers 
$1\le r_1<\dots <r_{i-d}\le i$ such that 
\[
\{r_1,\dots,r_{i-d}\}\cup\{l_1,\dots,l_d\}=\{1,\dots,i\}.
\]

\begin{lem}\label{det}
For $L\subset \{1,\dots,i,j+1,\dots,n\}$, the Pl\"ucker coordinate $X^{(i,j)}_L$ 
of the point $\exp(\sum_{a\le b} c_{a,b} f_{a,b}) [v_{\bf 1}]$ is given by
\[
(-1)^{\sum_{p=1}^d (l_p-p)}\det (c_{r_a,l_{d+b}-1})_{a,b=1}^{i-d}.
\]
\end{lem} 
\begin{proof}
We note that $f_{a,b} w_c=\delta_{a,c} w_{b+1}$. Therefore in $W_{i,j}$ one has:
\begin{multline*}
\exp(\sum_{a\le b} c_{a,b} f_{a,b}) w_1\wedge\dots\wedge w_i\\=
(w_1+c_{1,j}w_{j+1}+\dots +c_{1,n-1}w_n)\wedge\dots\wedge (w_i+c_{i,j}w_{j+1}+\dots + c_{i,n} w_n).
\end{multline*}
\end{proof}

\begin{example}
Let $d=i-1$, i.e. $l_{i-1}\le i$ and $l_i>j$. Let $\{r\}=\{1,\dots,i\}\setminus L$. Then
\begin{equation}\label{deg1}
X^{(i,j)}_L\left(\exp(\sum_{a\le b} c_{a,b} f_{a,b})[v_{\bf 1}]\right)=(-1)^{i-r} c_{r,l_i-1}.
\end{equation}
\end{example}

Let $\bC{\bf X}_n$ be the ring of multi-homogeneous polynomials in variables 
$X^{(i,j)}_L$, $1\le i\le j\le n-1$, $L=(l_1,\dots,l_i)$. In other words, $\bC{\bf X}_n$
is the coordinate ring of the product of projective spaces $\bP(W_{i,j})$, $1\le i\le j\le n-1$.
We fix the decomposition 
\[
\bC{\bf X}_n=\bigoplus_{\bm\in\bZ^{n(n-1)/2}} \bC{\bf X}_\bm,
\] 
where $\bC{\bf X}_\bm$ is the space of polynomials of total degree $m_{i,j}$ in all variables $X^{(i,j)}_L$.
For an ideal $I\subset \bC{\bf X}_n$ we denote by $X(I)\subset \prod_{i\le j} \bP(\Lambda^i W_{i,j})$ the  
variety of zeros of $I$ and for a subvariety $X$ in the product of projective spaces we denote by $I(X)$ the
ideal of multi-homogeneous polynomials vanishing on $X$.  

Our goal is to find the ideal $I(R_n)$ in  $\bC{\bf X}_n$ of multi-homogeneous polynomials vanishing on $R_n$.
Then the coordinate ring $Q_n$ of $R_n$ is given by 
$$Q_n=\bC{\bf X}_n/I(R_n).$$
We first describe the analogues of the Pl\"ucker relations contained in $I(R_n)$. 

The relations are labeled by the following data:
\begin{itemize}
\item two pairs $(i_1,j_1)$, $(i_2,j_2)$ with $i_1\ge i_2$ and $1\le i_1\le j_1\le n-1$, 
$1\le i_2\le j_2\le n-1$;
\item an integer $k$ satisfying $1\le k\le i_2$;
\item two collections of indices $L=(l_1,\dots,l_{i_1})$, $P=(p_1,\dots,p_{i_2})$ with 
$L\subset\{1,\dots,i_1,j_1+1,\dots,n\}$, $P\subset\{1,\dots,i_2,j_2+1,\dots,n\}$.
\end{itemize}
We denote the relation labeled by the data above by $R^{(i_1,j_1),(i_2,j_2);k}_{L,J}$. This is a polynomial
in the ring $\bC{\bf X}_n$ given by the following formulas:

First, let $i_1\ne i_2$ and $j_1\ge j_2$. Then 
\[
R^{(i_1,j_1),(i_2,j_2);k}_{L,P}=X^{(i_1,j_1)}_L X^{(i_2,j_2)}_P - 
\sum_{1\le r_1<\dots <r_k\le n} X^{(i_1,j_1)}_{L'} X^{(i_2,j_2)}_{P'},
\]
where the summation rums over $r_i$ satisfying $l_{r_i}\in\{1,\dots,i_2,j_1+1,\dots,n\}$ and for a term
$X^{(i_1,j_1)}_{L'} X^{(i_2,j_2)}_{P'}$ corresponding to the sequence $(r_1,\dots,r_k)$ one has
\begin{gather}\label{L'}
L'=(l_1,\dots,l_{r_1-1},p_1,l_{r_1+1},\dots,l_{r_k-1},p_k,l_{r_k+1},\dots,l_{i_1}),\\
\label{J'} P'=(l_{r_1},\dots,l_{r_k},p_{k+1},\dots,p_i).
\end{gather}

Second, let $j_1\le j_2$. Then we take the relations as above with an additional restriction 
$L\subset \{1,\dots,i_1,j_2+1,\dots,n\}$.

\begin{rem}
We note that the initial term $X^{(i_1,j_1)}_L X^{(i_2,j_2)}_J$ is present in the relations if and 
only if $\{p_1,\dots,p_k\}\subset \{1,\dots,i_2,j_1+1,\dots,n\}$.
\end{rem}

\begin{example} Let $n=3$. Then $\bC{\bf X}_3$ the ring of multi-homogeneous 
polynomials in three groups of variables 
\[
X^{(1,1)}_1, X^{(1,1)}_2, X^{(1,1)}_3;\quad X^{(1,2)}_1, X^{(1,2)}_3; \quad
X^{(2,2)}_{1,2}, X^{(2,2)}_{1,3}, X^{(2,2)}_{2,3}. 
\]
The Pl\"ucker relations are given by
\begin{gather*}
R^{(1,1),(1,2);1}_{(3),(1)}=X^{(1,1)}_3 X^{(1,2)}_1 - X^{(1,1)}_1 X^{(1,2)}_3,\\
R^{(2,2),(1,1);1}_{(2,3),(1)}=X^{(2,2)}_{2,3} X^{(1,1)}_1 + X^{(2,2)}_{1,2} X^{(1,1)}_3,\\
R^{(2,2),(1,2);1}_{(2,3),(1)}=X^{(2,2)}_{2,3} X^{(1,2)}_1 + X^{(2,2)}_{1,2} X^{(1,2)}_3.
\end{gather*} 
\end{example}

\begin{example}
Let $n=4$. Then the only relation containing variables $X^{(2,2)}_{l_1,l_2}$ and $X^{(1,3)}_{p}$
is
\[
R^{(2,2),(1,3);1}_{(1,2),(4)}=X^{(2,2)}_{1,2}X^{(1,3)}_4 + X^{(2,2)}_{2,4}X^{(1,3)}_1.
\]
\end{example}

\begin{example}\label{Gr}
Fix a pair $1\le i\le j\le n-1$. Then the ideal in $\bC[X^{(i,j)}_{l_1,\dots,l_i}]$ generated by the relations
$R^{(i,j),(i,j);k}_{L,P}$ is exactly the ideal defining the Grassmann variety $Gr_i(W_{i,j})$.
\end{example} 

\begin{prop}
The polynomials $R^{(i_1,j_1),(i_2,j_2);k}_{L,P}$ vanish on $R_n$.
\end{prop}

Let $J_n\subset \bC{\bf X}_n$ be the ideal generated by the Pl\"ucker relations  above.  
We will prove the following theorem.
\begin{thm}
The ideal $I(R_n)$ is the saturation of $J_n$. More precisely, for any $F\in I(R_n)$ there exist
non-negative integers $N_{i,j}$, $i\le j$, such that 
\[
\prod_{i\le j} \left(X^{(i,j)}_{1,\dots,i}\right)^{N_{i,j}} F\in J_n. 
\]
\end{thm}

We put forward the following conjecture.
\begin{conj}
The ideal $J_n$ is prime, i.e. $I(R_n)=J_n$.
\end{conj}

\section{Proofs.}
\subsection{The ideal}
\begin{prop}
$J_n\subset I(R_n)$, i.e. all the relations $R^{(i_1,j_1),(i_2,j_2),k}_{L,P}$ vanish on $R_n$.
\end{prop}
\begin{proof}
We recall that the Sylvester identity (see e.g. \cite{Fu}, section 8.1, Lemma 2) states that for any
two $s$ by $s$ matrices $M$ and $N$ and a number $k\le s$ one has 
\[
\det M \det N - \sum_{1\le r_1<\dots <r_k\le s} \det M' \det N'=0, 
\] 
where the matrices $M'$ and $N'$ are obtained from $M$ and $N$ by interchanging the first $k$ columns
of $N$ with the columns $r_1,\dots,r_k$ of $M$ (preserving the order of columns). Combining the Sylvester
identity with Lemma \ref{det} we obtain our proposition. 
\end{proof}

\begin{lem}\label{vert}
The relations $R^{(i,j),(i,j);k}_{L,P}$, $R^{(i+1,j),(i+1,j);k}_{L,P}$ and
$R^{(i,j),(i+1,j);k}_{L,P}$ with all possible $L$ and $P$ define the variety of pairs 
$(V_{i,j},V_{i+1,j})\subset Gr_i(W_{i,j})\times Gr_{i+1}(W_{i+1,j})$ satisfying $V_{i,j}\subset V_{i+1,j}$. 
\end{lem}
\begin{proof}
First, the relations $R^{(i,j),(i,j);k}_{L,P}$ and $R^{(i+1,j),(i+1,j);k}_{L,P}$ cut out the product of 
Grassmann varieties $Gr_i(W_{i,j})\times Gr_{i+1}(W_{i+1,j})$ inside the product of projective spaces
$\bP(\Lambda^i W_{i,j})\times \bP(\Lambda^{i+1} W_{i+1,j})$. Second, the relations $R^{(i,j),(i+1,j);k}_{L,P}$
are exactly the classical Pl\"ucker relations saying that the first subspace has to be a subset of the second.
\end{proof}

\begin{lem}\label{hor}
The relations $R^{(i,j),(i,j);k}_{L,P}$, $R^{(i,j+1),(i,j+1);k}_{L,P}$ and
$R^{(i,j),(i,j+1);k}_{L,P}$ with arbitrary $L$ and $P$ define the variety of pairs 
$(V_{i,j},V_{i,j+1})\subset Gr_i(W_{i,j})\times Gr_{i}(W_{i,j+1})$ satisfying $pr_{j+1}V_{i,j}\subset V_{i,j+1}$. 
\end{lem}
\begin{proof}
Similar to the proof of the lemma above. The only difference that the relations 
$R^{(i,j),(i,j+1);k}_{I,P}$ are degenerate Pl\"ucker relations (see \cite{Fe1}).
\end{proof}

\begin{prop}
We have $R_n=X(J_n)$, i.e. the common zeroes of $J_n$ give $R_n$.
\end{prop}
\begin{proof}
We need to prove that if all the relations $R$ vanish at a point $x\in \times_{i\le j} \bP(M_{\al_{i,j}})$,
then $x\in R_n$. Follows from Lemmas \ref{vert} and \ref{hor}.    
\end{proof}

The following proposition  is of the key importance for us.
\begin{prop}\label{G}
For an element $F\in \bC{\bf X}_n$ there exist numbers $N_{i,j}\in\bZ_{\ge 0}$, $i\le j$ 
and another polynomial $G\in \bC{\bf X}_n$ 
such that 
\[
\prod_{i\le j}(X^{(i,j)}_{1,\dots,i})^{N_{i,j}} F - G\in J_n
\] 
and $G$ depends only on the PBW-degree zero or one variables
\[
X^{(i,j)}_{1,\dots,i},\ i\le j\ \text{ and }\ X^{(i,j)}_{1,\dots,r-1,r+1,\dots,i,m},\ 1\le r\le i\le j<m,
\]  
In addition, $G$ can be chosen in such a way that if $G$ depends on a variable 
$X^{(i,j)}_{1,\dots,r-1,r+1,\dots,i,m}$ for some $i,j$ then it does not depend on all variables of the form
$X^{(i_1,j_1)}_{1,\dots,r-1,r+1,\dots,i_1,m}$ for $(i_1,j_1)\ne (i,j)$.
\end{prop}
\begin{proof}
Recall the PBW-degree of a variable $X^{(i,j)}_L$ given by the number of $l_a\in L$ such that $l_a>j$.
Let  $X^{(i,j)}_L$ be a variable of PBW-degree greater than one with $r\notin L$ for some $r\le i$. 
We consider the relation 
$R^{(i,j),(i,j);1}_{L,(r,1,\dots,r-1,r+1,\dots,i)}$. One has
\begin{multline*}
(-1)^{r-1}R^{(i,j),(i,j);1}_{L,(1,\dots,i)} = \\
X^{(i,j)}_L X^{(i,j)}_{(1,\dots,i)}-
\sum_{a:\ l_a>j} X^{(i,j)}_{(l_1,\dots,l_{a-1},r,l_{a+1},\dots,l_i)} X^{(i,j)}_{(1,\dots,r-1,l_a,r+1,\dots,i)}.
\end{multline*}
We note that the PBW-degree of each variable $X^{(i,j)}_{(l_1,\dots,l_{a-1},r,l_{a+1},\dots,l_i)}$
is one less than that of $X^{(i,j)}_L$. Hence, by decreasing induction, we obtain the first statement of
our proposition.

Now assume that a polynomial $G$ depends on the variables 
\[X^{(i_1,j_1)}_{1,\dots,r-1,r+1,\dots,i_1,m} \text{ and } X^{(i_2,j_2)}_{1,\dots,r-1,r+1,\dots,i_2,m}\]
for two different pairs $(i_1,j_1)$ and $(i_2,j_2)$. Let $i_1\ge i_2$. We consider the relation
\begin{multline*}
(-1)^{r-1}R^{(i_1,j_1),(i_2,j_2);1}_{(1,\dots,r-1,r+1,\dots,i_1,m),(r,1,\dots,r-1,r+1,\dots,i_2)}\\=
X^{(i_1,j_1)}_{1,\dots,r-1,r+1,\dots,i_1,m} X^{(i_2,j_2)}_{1,\dots,i_2} - 
(-1)^{i_2+i_1} X^{(i_1,j_1)}_{1,\dots,i_1} X^{(i_2,j_2)}_{1,\dots,r-1,r+1,\dots,i_2,m}.
\end{multline*}
Using these relations we can get rid of the variable 
$X^{(i_1,j_1)}_{1,\dots,r-1,r+1,\dots,i_1,m}$ (multiplying by $X^{(i_2,j_2)}_{1,\dots,i_2}$).
\end{proof}

\subsection{Representation of $Q_n$.} Let $Q_n=\bC{\bf X_n}/I(R_n)$ be the coordinate ring of $R_n$. 
Consider a polynomial algebra $A_n$ in variables $T_{i,j}$, 
$1\le i\le j\le n-1$ and $Z_{i,j}$, $1 \le i\le j\le n-1$. We define a homomorphism $\Psi:\bC{\bf X}_n\to A$
in the following way. Let $X^{(i,j)}_L$ be a variable with 
$$1\le l_1<\dots <l_d\le i\le j <l_{d+1}<\dots <l_i$$
(thus the PBW-degree of this variable is $i-d$).  We define $i\times i$ matrix $M$ by 
\begin{equation}\label{det1}
M_{a,b}= \begin{cases}
(-1)^{l_a-a}, \text{ if } b=a\le d, l_a\in L\\
Z_{a,l_b-1}, \text{ if } a\notin L, a\le i, b>d,\\
0, \text{ otherwise}.
\end{cases} 
\end{equation}
We define the homomorphism of polynomial algebras by the formula 
\begin{equation}\label{map}
X^{(i,j)}_{l_1,\dots,l_i}\mapsto T_{i,j} \det M.
\end{equation} 

\begin{prop}\label{vanish}
The kernel of $\Psi$ coincides with $I(R_n)$. 
\end{prop}
\begin{proof}
Recall the open dense cell $U\subset R_n$, which is the $(N^-)^a$-orbit through the highest weight line. 
We note that $I(R_n)=I(U)$. Now  the embedding
\[
R_n\subset \prod_{i\le j} \bP(\Lambda^i W_{i,j})
\]
is defined using the same determinants as in \eqref{det1} (see Lemma \eqref{det}). 
Hence a polynomial in $\bC{\bf X}_n$
vanishes on $U$ if and only if belongs to the kernel of $\Psi$
(we note that the variables $T_{i,j}$ guarantee the multi-homogeneity). 
\end{proof}

\begin{cor}\label{emb}
We have an embedding $Q_n\to A_n$ defined by \eqref{map}.
\end{cor}

\begin{cor}
One has the decomposition
\[
Q_n=\bigoplus_{\bm\in \bZ_{\ge 0}^{n(n-1)/2}} Q_\bm,
\]
where $Q_\bm$ consists of polynomials of total degree $m_{i,j}$ in variables $X^{(i,j)}_L$.
\end{cor}
We note that this decomposition is very similar to the one for the classical flag varieties, where 
the coordinate ring decomposes into the direct sum of dual irreducible $\g$-modules.
Thus it is very natural to ask about the properties of the spaces $Q_\bm$.

\subsection{Cocyclicity}
Our goal in this subsection is to prove that there exists a natural action of the Lie algebra $(\fn^-)^a$ on
$Q_n$ such that each $Q_\bm$ is cocycilc with respect to the action of $(\fn^-)^a$ with a cocyclic vector
$$v_\bm^*=\prod_{i\le j} (X^{(i,j)}_{1,\dots,i})^{m_{i,j}}.$$

Let $\bm$ be a collection such that for some $i,j$ we have $m_{i,j}=1$ and $m_{c,d}=0$ for all 
$(c,d)\ne (i,j)$ (these are analogues of the fundamental weights). We denote the corresponding homogeneneous
component $Q_\bm$ by $Q_{i,j}$.

We first define the structure of $(\fn^-)^a$-module on $Q_{i,j}$. 
Note that  the space $Q_{i,j}$ has a basis $X^{(i,j)}_L$ with $L=(1\le l_1<\dots <l_i\le n)$ and
$L\subset \{1,\dots,i,j+1,\dots,n\}$. We identify $Q_{i,j}$ with the dual space $\Lambda^i(W_{i,j})$
by setting $X^{(i,j)}_L w_P=\delta_{L,P}$ for any $P=(1\le p_1<\dots <p_i\le n)$, where  
$w_P=w_{p_1}\wedge \dots \wedge w_{p_i}$. This endows $Q_{i,j}$ with the structure of $\g^a$-module, 
$Q_{i,j}^*\simeq \Lambda^i(W_{i,j})$. In particular, one has
\[
f_{a,b} X^{(i,j)}_L=
\begin{cases}
-X^{(i,j)}_{l_1,\dots,l_{r-1},a,l_{r+1},\dots,l_i}, \text{ if } a\le i\le j<b \text{ and } l_r=b+1,\\
0, \text{ otherwise}. 
\end{cases}
\]
(Recall that for $\sigma\in S_i$ we have $X_{\sigma L}=(-1)^\sigma X_L$). We note that $Q_{i,j}$ is
coclycic with respect to the action of $(\fn^-)^a$ and $X^{(i,j)}_{1,\dots,i}$ is a cocyclic vector.

Now, let us define the structure of $(\fn^-)^a$-module on all $Q_\bm$. By definition, there exists a surjective
map $\bigotimes_{i\le j} Q_{i,j}^{\T m_{i,j}}\to Q_\bm$. We have a natural structure of $(\fn^-)^a$-module
on the tensor product and hence we only need to show that $I(R_n)$ is $(\fn^-)^a$-invariant.
We first prove this for $J_n$.

\begin{lem}
The ideal $J_n$ is invariant with respect to the $(\fn^-)^a$-action defined above.
\end{lem} 
\begin{proof}
It suffices to prove that for any element $f_{a,b}\in (\fn^-)^a$ one has 
$$f_{a,b} R^{(i_1,j_1),(i_2,j_2);k}_{L,P}\in J_n.$$ But it is easy to see that 
$f_{a,b} R^{(i_1,j_1),(i_2,j_2);k}_{L,P}$ is again (different) generalized Pl\"ucker relation from our list.
\end{proof}

\begin{thm}
The ideal $I(R_n)$ is $(\fn^-)^a$-invariant and it is a saturation of $J_n$.
\end{thm}
\begin{proof}
We first prove the second statement. 
Using Proposition \ref{G} we find numbers $N_{i,j}$ and a polynomial $G$ such that
\begin{equation}\label{FG}
\prod_{i\le j} (X^{(i,j)}_{1,\dots,i})^{N_{i,j}} F - G\in J_n
\end{equation}
and $G$ depends on the variables of PBW-degree at most $1$. In addition, only one variable of the form
$X^{(i,j)}_{1,\dots,r-1,r+1,\dots,i,m}$ shows up in $G$ for any pair $r<m$. We claim that if $F\in I(R_n)$
then $G=0$. In fact, we choose  pairs $r,m$ and $i,j$ with $r\le i\le j<m$ such that $G$ depends on 
$X^{(i,j)}_{1,\dots,r-1,r+1,\dots,a,m}$. Let 
\[
G=\sum_{p\ge 0} \left(X^{(i,j)}_{1,\dots,r-1,r+1,\dots,i,m}\right)^p G_p
\] 
be the decomposition with respect to powers of $X^{(i,j)}_{1,\dots,r-1,r+1,\dots,i,m}$
(i.e. $G_p$ are independent of $X^{(i,j)}_{1,\dots,r-1,r+1,\dots,i,m}$). Proposition \ref{vanish}
gives that
$\Psi G=0$. We note that $\Psi X^{(i,j)}_{1,\dots,r-1,r+1,\dots,i,m}=(-1)^{i-r} T_{i,j}Z_{r,m-1}$. Therefore,
the variable $Z_{r,m-1}$ comes only from $X^{(i,j)}_{1,\dots,r-1,r+1,\dots,i,m}$. We conclude that
all $G_p=0$. Continuing this procedure, we arrive at $G=0$. Therefore, \eqref{FG} now reads as
\[
\prod_{i\le j} (X^{(i,j)}_{1,\dots,i})^{N_{i,j}} F \in J_n
\]
and thus $I(R_n)$ is the saturation of $J_n$ (see \cite{H}, Exercise 5.10). 

In order to prove the first statement of the proposition 
we note that for all $a,b,i,j$ one has $f_{a,b} X^{(i,j)}_{1,\dots,i}=0$. Since $J_n$ is $\g^a$-invariant,
we obtain
\[
\prod_{i\le j} (X^{(i,j)}_{1,\dots,i})^{N_{i,j}} (f_{a,b} F) \in J_n\subset I(R_n).
\]
Now since $I(R_n)$ is simple, we obtain $f_{a,b} F\in I(R_n)$.
\end{proof}

Thanks to the theorem above, each $Q_\bm$ carries the structure of  $(\fn^-)^a$-module. Our next goal is to prove
that it is cocyclic with cocyclic vector being the product of $X^{(i,j)}_{1,\dots,i}$, i.e for any nontrivial
$F\in Q_\bm$ one has
\[
\prod_{i\le j} (X^{(i,j)}_{1,\dots,i})^{m_{i,j}}\in \bC[f_{a,b}]_{a\le b} F.
\]

\begin{thm}
$Q_\bm$ is cocyclic with a coclycic vector $\prod_{i\le j} (X^{(i,j)}_{1,\dots,i})^{m_{i,j}}$.
\end{thm}
\begin{proof}
Let us denote by $v_\bm^*$ the image of $\prod_{i\le j} (X^{(i,j)}_{1,\dots,i})^{m_{i,j}}$ in $Q_n$. 
Let $F\in \bC{\bf X}_n$ and $F\notin I(R_n)$. 
We use Proposition \ref{G} and write 
\[
\prod_{i\le j} (X^{(i,j)}_{1,\dots,i})^{N_{i,j}} F - G\in J_n
\]
for some numbers $N_{i,j}$ and a polynomial $G$. 
Since $F\notin I(R_n)$ we have $G\ne 0$. Suppose  that a term
\[
\prod_{i\le j} (X^{(i,j)}_{1,\dots,i})^{l_{i,j}}
\prod_{r<m} (X^{(i(r,m),j(r,m))}_{1,\dots,r-1,r+1,\dots,i(r,m),m})^{p(r,m)}
\]
with some indices $i(r,m)\le j(r,m)$ and powers $p(r,m)$  shows up in $G$ with a non-zero coefficient. Then 
since $f_{a,b} X^{(i,j)}_{1,\dots,i}=0$, we have 
\[
\prod_{r<m} f_{r,m}^{p(r,m)} G = \prod_{i\le j} (X^{(i,j)}_{1,\dots,i})^{l_{i,j}}
\prod_{r<m} f_{r,m}^{p(r,m)} (X^{(i(r,m),j(r,m))}_{1,\dots,r-1,r+1,\dots,i(r,m),m})^{p(r,m)}
\]
and the result is proportional (with a non-zero coefficient) to the product of powers of $X^{(i,j)}_{1,\dots,i}$.
Hence we arrive at the following identity: there exist numbers $N_{i,j}$, $p(r,m)$ such that
\[
\prod_{r<m} f_{r,m}^{p(r,m)} \left(\prod_{i\le j} (X^{(i,j)}_{1,\dots,i})^{N_{i,j}} F \right) =
\mathrm{const}\cdot \prod_{i\le j} (X^{(i,j)}_{1,\dots,i})^{N_{i,j}+m_{i,j}}.
\]
Since all $f_{a,b}$ annihilate $X^{(i,j)}_{1,\dots,i}$ we obtain
\[
\prod_{i\le j} (X^{(i,j)}_{1,\dots,i})^{N_{i,j}} \prod_{r<m} f_{r,m}^{p(r,m)} F - 
 \prod_{i\le j} (X^{(i,j)}_{1,\dots,i})^{N_{i,j}+m_{i,j}} \in I(R_n).
\]  
This gives 
\[
\left(\prod_{r<m} f_{r,m}^{p(r,m)}\right) F - 
\prod_{i\le j} (X^{(i,j)}_{1,\dots,i})^{m_{i,j}}\in I(R_n)
\]
and hence $(\prod_{r<m} f_{r,m}^{p(r,m)}) F=v_\bm^*$ in $Q_n$. 
\end{proof}

\begin{cor}
The modules $Q_\bm$ are cocyclic with respect to the abelian algebra $(\fn^-)^a$ with cocyclic
vectors $v_\bm^*$. One has the isomorphism of $(\fn^-)^a$-modules $Q_\bm^*\simeq M_\bm$.
\end{cor}
\begin{proof}
It is easy to see that the $(\fn^-)^a$-modules $Q^*_{i,j}$ and $M_{i,j}$ are isomorphic.
Since both families $Q_\bm^*$ and $M_\bm$ can be constructed inductively via the $(\fn^-)^a$-embeddings
\[
M_{\bm^1+\bm^2}\subset M_{\bm^1}\T M_{\bm^2},\
Q^*_{\bm^1+\bm^2}\subset Q^*_{\bm^1}\T Q^*_{\bm^2},
\] 
our corollary follows.
\end{proof}


\section{Conjectures and further directions.}
In this section we collect conjectures concerning geometric and algebraic objects described in the present paper.

For the readers convenience we first recall the conjecture about the Pl\"ucker relations already stated in the 
previous section. Recall the ideal $J_n$ generated by the (generalized) Pl\"ucker relations.

\begin{conj}
The ideals $J_n$ and $I(R_n)$ coincide.
\end{conj}
One way to prove such kind of statements is to find a set of linearly independent monomials in 
$Q_n=\bC{\bf X}_n/I(R_n)$
such that these monomials form a spanning set for $\bC{\bf X}_n/J_n$ (recall $J_n\subset I(R_n)$).
This is usually done in terms of some kind of (semi)standard tableaux (see \cite{Fu}, \cite{Fe1}).
It is interesting to describe the corresponding combinatorics in our situation.

The following conjecture provides a presentation of the modules $M_\bm$ in terms of generators and relations. 
It generalizes similar statements from \cite{FFoL1}, where the case of $\bm$ supported on the diagonal was considered ($m_{i,j}=0$ unless $i=j$).
\begin{conj}
The following $\g^a$-module $M_\bm$ is isomorphic to the quotient $\bC[f_{i,j}]_{i\le j}/I_\bm$, where
$I_\bm$ is the ideal generated by the subspace 
\begin{equation}\label{conj1}
\U(\fb) \circ \mathrm{span}(f_{i,j}^{\sum_{i\le k\le l\le j} m_{k,l}+1}, i\le j). 
\end{equation}
\end{conj}

We also conjecture that the monomial basis of \cite{FFoL1}, \cite{V} can be extended to our case. Namely, let 
$S_\bm$ be the subset of $\bZ_{\ge 0}^{n(n-1)/2}$ consisting of collections $\bs=(s_\al)_{\al>0}$ 
such that for any Dyck path $\bp$ starting at $\al_i$ and ending at $\al_j$ one has
\[
\sum_{\beta\in\bp} s_\beta\le \sum_{i\le k\le l\le j} m_{k,l}.
\]
\begin{conj}
The elements $\{f^\bs v_\bm, \bs\in S_\bm\}$ form a basis of $M_\bm$.
\end{conj}

We note that two conjectures above are closely related. Namely, it is easy to show that there is a surjection
$\bC[f_{i,j}]_{i\le j}/I_\bm\to M_\bm$ (because the relations of the left hand side
hold in $M_\bm$). In addition we have the following lemma:
\begin{lem}
The vectors $\{f^\bs, \bs\in S_\bm\}$ span the right hand side of \eqref{conj1}. 
\end{lem}
\begin{proof}
The proof is very similar to the one in \cite{FFoL1}.
\end{proof}

Hence in order to prove two conjectures above it suffices to show that the vectors 
$\{f^\bs v_\bm, \bs\in S_\bm\}$ are linearly independent. Computer experiments support this conjecture, but we 
are not able to prove it so far. The main difference with the case of \cite{FFoL1} is that the the Minkowsky
sum $S_{\bm^1}+S_{\bm^2}$ is no longer equal to $S_{\bm^1+\bm^2}$ for general $\bm^1$ and $\bm^2$.

\section*{Acknowledgments}
This work was partially supported
by the Russian President Grant MK-3312.2012.1, by the Dynasty Foundation and 
by the AG Laboratory HSE, RF government grant, ag. 11.G34.31.0023.
This study comprises research findings from the `Representation Theory
in Geometry and in Mathematical Physics' carried out within The
National Research University Higher School of Economics' Academic Fund Program
in 2012, grant No 12-05-0014.
This study was carried out within "The National Research University Higher School of Economics' Academic Fund Program in 2012-2013, research grant No. 11-01-0017.

\end{document}